\documentclass[12pt,reqno]{article}

\usepackage[usenames]{color}
\usepackage{amssymb}
\usepackage{amsmath}
\usepackage{amsthm}
\usepackage{amsfonts}
\usepackage{amscd}
\usepackage{graphicx}

\usepackage[colorlinks=true,
linkcolor=webgreen,
filecolor=webbrown,
citecolor=webgreen]{hyperref}

\definecolor{webgreen}{rgb}{0,.5,0}
\definecolor{webbrown}{rgb}{.6,0,0}

\usepackage{color}
\usepackage{fullpage}
\usepackage{float}

\usepackage{graphics}
\usepackage{latexsym}
\usepackage{epsf}

\DeclareMathOperator{\ce}{ce}
\DeclareMathOperator{\rt}{rt}
\DeclareMathOperator{\RT}{RT}
\DeclareMathOperator{\auts}{{\cal A}}
\def\ra{\rightarrow}
\def\modd#1 #2{#1\ \mbox{\rm (mod}\ #2\mbox{\rm )}}

\begin{document}

\theoremstyle{plain}
\newtheorem{theorem}{Theorem}
\newtheorem{corollary}[theorem]{Corollary}
\newtheorem{lemma}[theorem]{Lemma}
\newtheorem{proposition}[theorem]{Proposition}

\theoremstyle{definition}
\newtheorem{definition}[theorem]{Definition}
\newtheorem{example}[theorem]{Example}
\newtheorem{conjecture}[theorem]{Conjecture}

\theoremstyle{remark}
\newtheorem{remark}[theorem]{Remark}

\title{Repetition Threshold for Binary Automatic Sequences
}
\author{
J.-P. Allouche\\
CNRS, IMJ-PRG, Sorbonne, \\
4 Place Jussieu, F-75252 Paris Cedex 05, France\\
\href{mailto:jean-paul.allouche@imj-prg.fr}{\tt jean-paul.allouche@imj-prg.fr}
\and N. Rampersad\\
Department of Mathematics and Statistics,
University of Winnipeg,\\ 
Winnipeg, MB R3B 2E9, Canada \\
\href{mailto:n.rampersad@uwinnipeg.ca}{\tt n.rampersad@uwinnipeg.ca} 
\and J. Shallit\\
 School of Computer Science, University of Waterloo, \\
 Waterloo, Ontario N2L 3G1, Canada\\
\href{mailto:shallit@uwaterloo.ca}{\tt shallit@uwaterloo.ca}}
\maketitle
\begin{abstract}
The critical exponent of an infinite word $\bf x$ is the supremum, over all finite nonempty factors $f$, of the exponent of $f$.
In this note we show that for all integers $k\geq 2$, there is a binary infinite $k$-automatic sequence with critical exponent $\leq 7/3$.  The same conclusion holds for Fibonacci-automatic and Tribonacci-automatic sequences.
\end{abstract}

\section{Introduction}
Let $w = w[1..n]$ be a finite word of length $n$.
We say $w$ has period $p$, for $1 \leq p \leq n$,
if $w[i] = w[i+p]$ for $1 \leq i \leq n-p$.
We define the exponent of $w$,
$\exp(w)$, to be $n/p$, where $p$ is the smallest
period of $w$.  Thus, for example, the French word
{\tt entente} has periods $3,6,7$ and exponent
$7/3$.

Let $x$ be a finite or infinite word, and $\alpha > 1$ be real number. If $x$ has no finite factors of exponent $\geq \alpha$, we say it is $\alpha$-free.   If it has no finite factors of exponent $>\alpha$, we say it is
$\alpha^+$-free.
We define
the critical exponent of $\bf x$, written
$\ce({\bf x})$, to be the supremum, over all
finite nonempty factors $w$ of $\bf x$, of 
$\exp(w)$.   

Let $S$ be a nonempty set of infinite words (or infinite sequences; we use
the two terms indifferently).
We define the repetition threshold 
$\rt(S)$ to be the infimum, over all elements 
$\bf x$ of $S$, of $\ce({\bf x})$.  Determining
the repetition threshold of a class of words is an intriguing and often difficult problem.   In its most
famous incarnation, we take $S$ to be the set of
all infinite words over a $k$-letter alphabet $\Sigma_k = \{0,1,\ldots, k-1\}$.
In this case we define $\RT(k) = \rt(\Sigma_k^\omega)$.

Thue \cite{Thue:1906,Thue:1912} showed that
$\RT(2) = 2$.   The classical example of a word achieving this bound
is ${\bf t} = 01101001\cdots$, the Thue-Morse sequence, which is the infinite fixed point of the morphism
$0 \rightarrow 01$, $1 \rightarrow 10$.  This word contains words of exponent $2$, but none of any higher
exponent. Dejean proved that $\RT(3) = 7/4$ and
gave a very influential conjecture about the values of
$\RT(k)$ for $k \geq 4$.   This conjecture was studied by many researchers, and finally resolved positively
by Currie and Rampersad \cite{Currie&Rampersad:2011} and Rao \cite{Rao:2011}, independently.

Instead of studying the repetition threshold for all infinite words over a $k$-letter alphabet, we can also study various subsets of interest.   For example, 
this was done for the case of 
Sturmian words \cite{Carpi&deLuca:2000}, 
circular words \cite{Gorbunova:2012},
balanced words 
\cite{Rampersad&Shallit&Vandomme:2019}, 
sequences rich in palindromes \cite{Currie&Mol&Rampersad:2020}, and
episturmian words
\cite{Dvorakova&Pelantova:2024}.

In this note we consider the repetition threshold for a very interesting class of sequences, the automatic sequences.   A sequence $(a_n)_{n \geq 0}$ is $k$-automatic if there is a deterministic finite automaton with output (DFAO) that, on input $n$ expressed in base $k$, computes $a_n$.  In  this  paper, all automata read the representation of numbers starting with the {\it most\/} significant digit.
We let $\auts_k$ denote the set of all binary $k$-automatic sequences.  For more about these sequences, see \cite{Allouche&Shallit:2003}.  For example, the Thue-Morse sequence is $2$-automatic.

We can also consider a more general notion, the 
{\it generalized automatic sequences}, where base $k$ is replaced by
some linear numeration system.
For each numeration system, one can ask for the corresponding repetition threshold for binary sequences.   In general, these problems seems quite difficult.  
However, if the numeration system is {\it addable}---that is, the addition
relation is computable by a finite automaton---then we have more tools
at our disposal.  
Examples of this type of numeration system
include the Zeckendorf system based
on Fibonacci numbers, and an analogous system based on the
Tribonacci numbers.

In this note we focus on the case of alphabet size $2$. Here we have a useful tool, a theorem that
characterizes $\alpha$-free finite and infinite words
\cite{Karhumaki&Shallit:2004}.
There are three main results:  Theorem~\ref{thm1}, which exactly characterizes the repetition threshold for ${\cal A}_k$, $k \geq 2$; Theorem~\ref{thm11}, which does the same thing for the binary Fibonacci-automatic sequences; and Theorem~\ref{thm12} for Tribonacci-automatic sequences.

\section{Results for \texorpdfstring{$k$}\ -automatic sequences}

For us a morphism is a map
$h:\Sigma_i^* \rightarrow \Sigma_j^*$
satisfying the identity
$h(xy) = h(x)h(y)$ for $x, y \in \Sigma_i^*$.
If the image of each letter is of length
$k$, we say that $h$ is $k$-uniform.   A $1$-uniform morphism is called a coding.   A morphism $h$ is called {\it squarefree\/} if
$h(x)$ is squarefree for all squarefree $x$. 

If $h(a) = ax$ for some letter $a$,
and $h^t(x) \not= \epsilon$ for all $t$,
then $h$ generates an infinite word
starting with $a$ by iteration:
$h^\omega(a) = \lim_{t\rightarrow\infty} h^t(a)$.  Notice that a morphism can generate a squarefree word by iteration, without the morphism itself being squarefree.

Our first main result characterizes the repetition threshold for binary $k$-automatic sequences. 
\begin{theorem}
\leavevmode
\begin{itemize}
    \item[(a)] If $k$ is a power of $2$, then
    $\rt(\auts_k) = 2$.

    \item[(b)] If $k$ is not a power of $2$, then
    $\rt(\auts_k) = 7/3$.
\end{itemize}
\label{thm1}
\end{theorem}

We will need a few lemmas.
\begin{lemma}
The sequence ${\bf a}$ is the image, under a coding, of a fixed point
of a $k$-uniform morphism if and only if $\bf a$ is 
$k$-automatic.
\label{lemma0}
\end{lemma}

\begin{proof}
See \cite[Theorem 3]{Cobham:1972} or \cite[Theorem 6.3.2]{Allouche&Shallit:2003}.
\end{proof}

\begin{lemma}
Let ${\bf a}$ be a $k$-automatic sequence
over $\Sigma_k$, and let $h$ be a uniform morphism
from $\Sigma_k^*$ to $\Sigma_\ell^*$.   Then
$h({\bf a})$ is also $k$-automatic.
\label{lemma1}
\end{lemma}

\begin{proof}
See \cite[Corollary 6.8.3]{Allouche&Shallit:2003}.
\end{proof}

\begin{lemma}
Let $i, j \geq 1$ be integers and $k \geq 2$.   Then a sequence is
$k^i$-automatic if and only if it is 
$k^j$-automatic.
\label{lemma3}
\end{lemma}

\begin{proof}
See \cite{Cobham:1969} or \cite[Theorem 6.6.4]{Allouche&Shallit:2003}.
\end{proof}

\begin{lemma}
Let ${\bf a}$ be a squarefree sequence over the
alphabet $\Sigma_4$, and define 
\begin{align*}
h(0) &= 011010011001001101001;\\
h(1) &= 100101100100110010110;\\
h(2) &= 100101100110110010110;\\
h(3) &= 011010011011001101001.
\end{align*}
Then $h({\bf a})$ is $(7/3)^+$-free.
\label{lemma2}
\end{lemma}

\begin{proof}
See \cite[Lemma 8]{Karhumaki&Shallit:2004}.
\end{proof}

\begin{lemma}
There is a $k$-uniform squarefree morphism over $\Sigma_3$
for all $k \geq 13$, except for $k = 14,15,16,20,21,22$.
\label{lemma4}
\end{lemma}

\begin{proof}
See Currie \cite{Currie:2013}.
\end{proof}

\begin{lemma}\label{lem:struct73}
Let ${\bf x}$ be an infinite $7/3$-free word over $\{0,1\}$.
Then ${\bf x} = p\mu({\bf y})$ for some $p \in \{\epsilon,0,00,1,11\}$
and some infinite $7/3$-free word ${\bf y}$.
\end{lemma}

\begin{proof}
This is \cite[Theorem~1]{Rampersad&Shallit&Shur:2011}, which is in turn
a consequence of the finitary result \cite[Theorem~6]{Karhumaki&Shallit:2004}.
\end{proof}

\begin{corollary}\label{cor:large_tm_blocks}
Let ${\bf x}$ be an infinite $7/3$-free word over $\{0,1\}$.
Then ${\bf x}$ contains the factors $\mu^i(0)$ and $\mu^i(1)$
for all $i \geq 0$.
\end{corollary}

\begin{proof}
This follows by iteration of Lemma~\ref{lem:struct73}.
\end{proof}

\begin{lemma}\label{lem:lexleast73}
The lexicographically least infinite $7/3$-free word over $\{0,1\}$ starting with $1$ is
$\overline{\bf t}$ (i.e., the complement of the Thue--Morse word ${\bf t}$).
\end{lemma}

\begin{proof}
Let ${\bf x}$ be the lexicographically least infinite $7/3$-free word over $\{0,1\}$ starting with $1$.
By Lemma~\ref{lem:struct73}, we have
\begin{align*}
{\bf x} &= 1001011001101001 \cdots\\
&= \mu(10010110\cdots)\\
&= \mu({\bf x'}),
\end{align*}
where ${\bf x'}$ is $7/3$-free.  However, the morphism $\mu$ preserves  lexicographic order,
so we must have ${\bf x'}={\bf x}$.  Hence, we have ${\bf x}=\mu({\bf x})$,
which implies that ${\bf x}=\overline{\bf t}$.
\end{proof}

For more on lexicographically extremal $7/3$-free words, see \cite{Rampersad&Shallit&Shur:2011}.

\bigskip

We recall the notion of {\it orbit closure\/} of an infinite sequence $\bf x$:
this is the set of all infinite sequences $\bf y$
such that every finite prefix of $\bf y$ appears somewhere in $\bf x$.
\begin{lemma}\label{lem:least_orbit_closure}
Let ${\bf x}$ be a generalized automatic sequence in an addable
numeration system and let $w$ be a finite word.
The lexicographically least sequence
in the orbit closure of ${\bf x}$ beginning with $w$ is also a generalized automatic sequence for the same numeration
system as ${\bf x}$.
\end{lemma}

\begin{proof}
This can be shown using {\tt Walnut} (a theorem-prover 
for automatic sequences \cite{Mousavi:2016,Shallit:2023}) 
by an easy modification of the approach described in \cite{Allouche&Rampersad&Shallit:2009} or
\cite[Theorem~8.11.1]{Shallit:2023}.   
\end{proof}

We can now prove our first main result, Theorem~\ref{thm1}.
\begin{proof}
\leavevmode

(a) If $k$ is a power of $2$, then by Lemma~\ref{lemma3}, a $2$-automatic sequence is also a $2^i$-automatic sequence for all $i \geq 1$.  So it suffices to prove the result for $k = 2$.  But then $\bf t$ is a $2$-automatic sequence that avoids overlaps, and it is easy to see that no binary word of length $>3$ can avoid squares.

\bigskip

(b) Suppose $k$ is not a power of $2$.   Suppose
there is some ${\bf x} \in \auts_k$ with
$\ce({\bf x}) = \alpha < 7/3$.   
Then, by Corollary~\ref{cor:large_tm_blocks}, we know that 
$\bf x$ contains $\mu^i(0)$ as a factor for arbitrarily large $i$.   
Therefore the orbit closure of $\bf x$
contains all sequences in the orbit closure of $\bf t$.
By Lemma~\ref{lem:lexleast73}, the lexicographically smallest $7/3$-free sequence starting with $1$ is
$\overline{\bf t}$, and hence the lexicographically smallest sequence in the orbit closure of ${\bf x}$ starting with $1$ is also
$\overline{\bf t}$.  It now follows from Lemma~\ref{lem:least_orbit_closure} that $\overline{\bf t}$ is $k$-automatic.
Then, by a theorem of Cobham \cite{Cobham:1969}, we know 
(since $k$ and $2$ are multiplicatively independent) that 
$\overline{\bf t}$ must be ultimately periodic, a 
contradiction.

Therefore $\ce({\bf x}) \geq 7/3$.   It remains to show that there is a word of critical exponent $7/3$ for all the relevant $k$.

By combining Lemmas~\ref{lemma4}, \ref{lemma2}, and \ref{lemma1}, we see that the desired result holds
for all $k \geq 23$, and also for $k=13,17,18,19$.  
Thus, the only remaining cases are $$k=3,5,6,7,9,10,11,12,14,15,20,21,22.$$
We now treat
each of these cases. The strategy is to find, for each 
of these values of $k$, a $k$-uniform morphism over $\Sigma_4$ whose fixed point
is squarefree, and then apply Lemma~\ref{lemma2}.
In some cases, the morphism we find is actually over
$\Sigma_3$.   It is not necessary to treat $k = 9$,
since it is implied by the result for $k = 3$ by Lemma~\ref{lemma3}.

We found the appropriate morphisms through either breadth-first or depth-first search.   
Table~\ref{tab1} gives the images of the morphisms we found.
\begin{table}[H]
    \centering
    \resizebox{.25\columnwidth}{!}{%
    \begin{tabular}{c|r}
       $k$ & morphism \\
        \hline
       3 & $0 \ra 021$ \\
         & $1 \ra 031$ \\
         & $2 \ra 012$ \\
         & $3 \ra 013$\\
        \hline
        5 & $0 \ra 01321$ \\
          & $1 \ra 01231$ \\
          & $2 \ra 02031$ \\
          & $3 \ra 02321$ \\
        \hline
        6 & $0 \ra 013121$ \\
          & $1 \ra 013031$ \\
          & $2 \ra 013231$ \\
          & $3 \ra 012131$ \\
        \hline
        7  & $0 \ra 0123021$ \\
           & $1 \ra 0123121$ \\
           & $2 \ra 0120321$ \\
           & $3 \ra 0121321$ \\
          \hline
        10 & $0 \ra 0120131021 $ \\
 & $1 \ra 0120231021 $ \\
 &  $2 \ra 0120301021 $ \\
 & $3 \ra 0120312021$ \\
 \hline
       11 & $0 \ra 01201020121 $ \\
 & $1 \ra 01202120121 $ \\
& $2 \ra 02010212021 $\\
\hline 
12 & $0 \ra 012021201021$ \\
& $1 \ra 012010212021 $ \\
& $2 \ra 012102120121$\\
\hline
14 & $0 \ra 01201023201021 $ \\
& $1 \ra 
 01201030201021 $ \\
& $2 \ra 
 01201031201021 $ \\
& $3 \ra
 01201020301021 $\\
 \hline
 15 & $0 \ra 012010230201021 $ \\
 & $1 \ra 012010231201021 $ \\
 & $2 \ra 012010203201021 $ \\
 & $3 \ra 012010213201021$\\
 \hline
 20 & $0 \ra 01201020120210201021 $ \\
 &  $1 \ra 01201021201210201021 $ \\
 & $2 \ra 01201020121021201021$ \\
 \hline
 21 & $0 \ra 012010201230210201021 $ \\
 & $1 \ra 012010201231210201021 $ \\
 & $2 \ra 012010201203210201021 $ \\
 & $3 \ra 012010201213210201021$\\
 \hline
 22 & $0 \ra 0120102012030210201021$ \\
 & $1 \ra 0120102012130210201021 $\\
 &  $2 \ra 0120102012320210201021 $ \\
 & $3 \ra 0120102012301210201021 $
 \end{tabular}
 }
    \caption{Morphisms generating squarefree words by iteration.}
    \label{tab1}
\end{table}

We can prove that each of these words are squarefree using the following model for $k =3$:
\begin{verbatim}
morphism ff3 "0->021 1->031 2->012 3->013":
promote FF3 ff3:
eval test3 "?msd_3 ~Ei,n (n>0) & Au,v (u>=i & u<i+n & v=u+n) => FF3[u]=FF3[v]":
\end{verbatim}
And {\tt Walnut} returns {\tt TRUE} for all of them.
\end{proof}









Since the construction involving the morphism $h$ does not necessarily produce an automatic sequence with the smallest possible DFAO for a $k$-automatic $(7/3)^+$-free infinite binary word,
it may be of interest to try to determine the number of states in a smallest DFAO.  
Lower bounds can be found by breadth-first search; upper bounds can be found by applying the construction involving $h$ or other means.
Table~\ref{tab2} summarizes what we currently know.  An asterisk denotes those entries where
the exact value is known; intervals denote
known lower and upper bounds.  The automata can be constructed straightforwardly from the expression as the iteration of a fixed point starting with $0$, followed by the
specified coding.
\begin{table}[htb]
    \centering
    \resizebox{.41\columnwidth}{!}{%
    \begin{tabular}{c|c|l}
       & number of &  \\
       $k$ & states $s$ & automaton \\
       \hline
       $2$ & $2$* & $0 \ra 01$ \\
                 && $1 \ra 10$\\
                \hline
       $3$ & $8 \leq s \leq 16$ & See below.\\[2pt]
       \hline
       $4$ & $2$* & $0 \ra 0110$ \\
                  && $1 \ra 1001$ \\
                \hline
       $5$ & $4$*  & $0 \ra 01230$ (0)\\
                   && $1 \ra 20120$ (0) \\
                   && $2 \ra 23012$ (1) \\
                   && $3 \ra 02302$ (1) \\
                   \hline
       $6$ & $6 \leq s \leq 27$ & Apply $h$ to word generated by  \\
       && morphism for $k=6$ in
       Table~\ref{tab1}.\\
       \hline
       $7$ & $4$*  & $0 \ra 0120330 $ (0)\\
                   &&$1 \ra 1211321$ (0) \\
                   &&$2 \ra 2112033$ (1) \\
                   &&$3 \ra 3013303$ (1) \\
                   \hline
       $8$ & $2$* & $0 \ra 01101001$ \\
                  && $1 \ra 10010110$ \\
                  \hline
       $9$ & $4$* & $0 \ra 012301312$ (0) \\
                  &&$1 \ra 312012312$ (0) \\
                  &&$2 \ra 301312312$ (1) \\
                  &&$3 \ra 013123012$ (1) \\
                  \hline
       $10$ & $4$* & $0 \ra 0120132132$ (0) \\
                  &&$1 \ra 0121321201$ (0) \\
                  &&$2 \ra 3212012132$ (1) \\
                  &&$3 \ra 3201213201$ (1) \\
                  \hline
       $11$ & $4$* & $0 \ra 01202301230$ (0) \\
                  && $1 \ra 20120230201$ (0) \\
                  && $2 \ra 23012012302$ (1) \\
                  && $3 \ra 30123020123$ (1)
    \end{tabular}
    }
    \caption{Minimum number of states
    needed to produce a $k$-automatic
    $(7/3)^+$-free infinite binary word.}
    \label{tab2}
    \end{table}

To get a $16$-state DFAO {\tt D3} for base $3$, use the following {\tt Walnut} commands:
\begin{verbatim}
morphism d9 "0->012301312 1->312012312 2->301312312 3->013123012":
promote D d9:
morphism code "0->0 1->0 2->1 3->1":
image D9 code D:
eval test9a "?msd_9 ~Ei,p (p>0) & Au (u>=i & 3*u<=3*i+4*p) => D9[u]=D9[u+p]":
# tests that morphism does indeed generate a (7/3)+-free word
convert D3 msd_3 D9: 
\end{verbatim}
These commands take a DFAO for a $9$-automatic word and compute a DFAO for inputs in base $3$.

\section{Fibonacci- and Tribonacci-automatic sequences}

In the Zeckendorf or Fibonacci numeration system \cite{Lekkerkerker:1952,Zeckendorf:1972},
we write every natural number $n$ uniquely as a sum of Fibonacci numbers $F_i$ for $i \geq 2$, subject to the condition that we never use both $F_i$ and $F_{i+1}$ in the sum. The representation can be viewed as a binary word, where we write $1$ if $F_i$ is included and $0$ otherwise.  A sequence $(a_n)_{n \geq 0}$ is said to be {\it Fibonacci-automatic} if there is a DFAO that on input $n$ expressed in the Zeckendorf numeration system reaches a state with output $a_n$ \cite{Mousavi&Schaeffer&Shallit:2016,Du&Mousavi&Rowland&Schaeffer&Shallit:2017,Du&Mousavi&Schaeffer&Shallit:2016}.  We let $\auts_F$ denote the set of all binary Fibonacci-automatic sequences.   The following is our second main result.
\begin{theorem}
We have $\rt(\auts_F) = 7/3$.
\label{thm11}
\end{theorem}

\begin{proof}
Let $\bf x$ be Fibonacci-automatic.
The possibility 
$\ce({\bf x}) < 7/3$ is ruled out by the same argument 
as in the proof of Lemma~\ref{lem:least_orbit_closure} (b) 
above.   
Namely, the automaton gives $\bf x$ 
as the image, under a coding, of a fixed 
point of some morphism $\varphi$ (as in Example~1 in
\cite[Section~5]{Shallit:1988}). As in the base-$k$ case,
the dominant 
eigenvalue of the incidence matrix of $\varphi$ is
$\alpha = \frac{1}{2}(1 + \sqrt{5})$.  
An easy induction proves that 
$\alpha^i = L_i/2 + F_i \sqrt{5}/2$, where $L_i$ is the $i$'th
Lucas number, so $\alpha$ and $2$ are multiplicatively independent.
Then a theorem of Durand \cite{Durand:2011}
finishes this part of the proof.

So it suffices to give an example achieving the 
exponent $7/3$. The easiest way to do this is to find a 
Fibonacci-automatic squarefree sequence $\bf y$ over a 
$4$-letter alphabet, and then apply the morphism $h$ to 
$\bf y$.  Such a sequence is computed by the DFAO in 
Figure~\ref{fig1}.
\begin{figure}[htb]
\begin{center}
\includegraphics[width=6in]{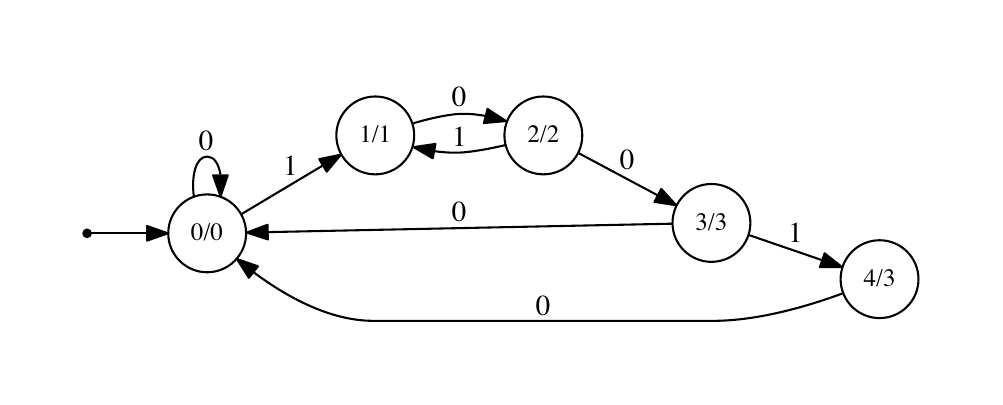}
\end{center}
\caption{Fibonacci automaton computing the sequence $\bf y$.}
\label{fig1}
\end{figure}
Once we have saved the DFAO in Figure~\ref{fig1} in the {\tt Word Automata} directory of {\tt Walnut} under the name {\tt Y.txt}, we can verify that the sequence it generates is squarefree, as follows:
\begin{verbatim}
eval ysquf "?msd_fib ~Ei,n n>=1 & At (t<n) => Y[i+t]=Y[i+t+n]":
\end{verbatim}
and {\tt Walnut} returns {\tt TRUE}.
\end{proof}

\begin{remark}
By the way, the critical exponent of the word {\bf y} is
$(5+\sqrt{5})/4 \doteq 1.809$,
although we do not use this fact anywhere.   The reader can compare {\bf y} to the word ${\bf x}_4$ in
\cite{Rampersad&Shallit&Vandomme:2019}.

When we apply the morphism $h$ to $\bf y$, we get a DFAO of 943 states.  It can be computed via the following {\tt Walnut} code:
\begin{verbatim}
morphism ks 
 "0->011010011001001101001 1->100101100100110010110 
  2->100101100110110010110 3->011010011011001101001":
image YP ks Y:
\end{verbatim}

This naturally suggests the question, is there a smaller Fibonacci DFAO computing a $(7/3)^+$-free binary sequence?

There is such a sequence ${\bf w} = 01101001101100110100 \cdots$ generated by a DFAO of $32$ states.  Let $x^R$ denote the reversal of the finite word $x$.
We can construct $\bf w$ iteratively as follows:  define
\begin{align*}
    w_0 &= 011 \\
    w_1 &= 01101 \\
    w_n &= \begin{cases}
        w_{n-1} \overline{w_{n-2}}^R, & \text{if $n \geq 2$ and $n \equiv \modd{1,2} {3}$}; \\[.1in]
        w_{n-1} w_{n-2}^R, & \text{if $n \geq \modd{0} {3}$}.
    \end{cases}
\end{align*}
Clearly each $w_i$ is a prefix of all $w_j$ for $j \geq i$, so we can define $\bf w$ as the unique infinite word of which all the $w_i$ are prefixes.  It is not difficult to show, with {\tt Walnut}, that this word is generated 
by a  Fibonacci DFAO with $32$ states and is
$(7/3)^+$-free.  We omit the details.
This means the size of the smallest such DFA is at most $32$.   Breadth-first search shows it is at least $12$.
\end{remark}


One can also carry out the same sorts of computations for the Tribonacci-automatic sequences, where we use a numeration system based on the Tribonacci numbers \cite{Mousavi&Shallit:2015}.  This gives our final main result. 
\begin{theorem}
The repetition threshold for binary Tribonacci-automatic sequences is $7/3$.
\label{thm12}
\end{theorem}

\begin{proof}
    The  proof is quite similar to the proof of Theorem~\ref{thm11}, so we just sketch the details.

    By breadth-first search, we can identify a candidate squarefree
    Tribonacci-automatic sequence $\bf z$ over $\Sigma_4$, and prove it is indeed squarefree using {\tt Walnut}.  It is generated by the DFAO in Figure~\ref{fig3}.
    \begin{figure}[H]
\begin{center}
\includegraphics[width=6in]{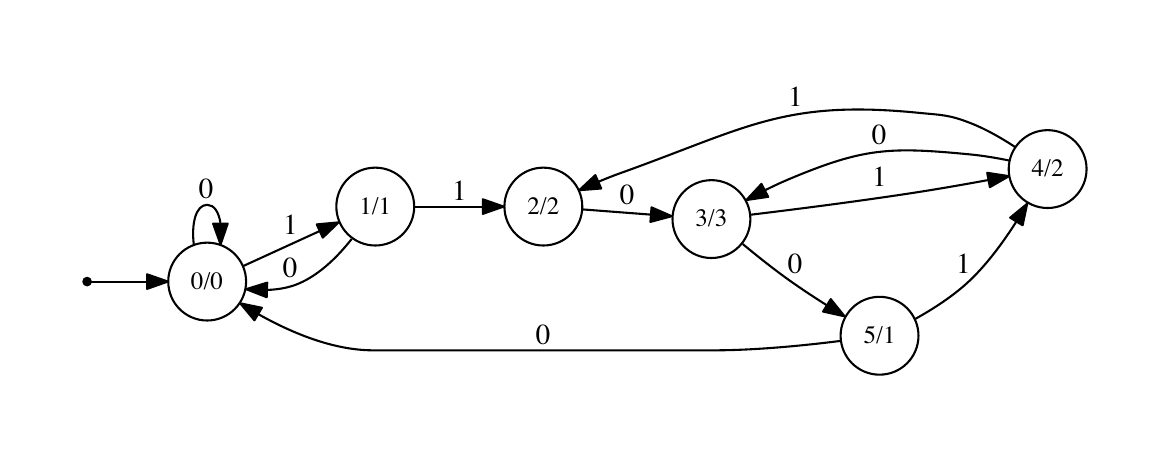}
\end{center}
\caption{Tribonacci automaton computing a squarefree sequence.}
\label{fig3}
\end{figure}
We can then apply the morphism $h$ to obtain a $(7/3)^+$-free binary Tribonacci-automatic sequence generated by a DFAO with 36849 states.  No binary sequence of smaller exponent is possible, by an argument similar to that in the proof of Theorem~\ref{thm11}.

\end{proof}

\begin{remark}
It can be shown that the critical exponent of the sequence $\bf z$ is $1.875954364915\cdots$, the
real zero of the polynomial
$4X^3-24X^2+56X-47$.
\end{remark}

\section{Going further}

It seems likely that $7/3$ is the critical exponent for many other binary
generalized automatic sequences, such as the Pell-automatic sequences \cite{Baranwal&Shallit:2019}, but a general theorem seems hard to prove.

\section*{Acknowledgments}
We gratefully acknowledge 
the support of the 
Natural Sciences and Engineering Research Council
of Canada (NSERC). Nous remercions le Conseil de recherches en sciences naturelles et en
g\'enie du Canada (CRSNG) de son soutien. 
The research of NR was supported by NSERC, grants RGPIN-2019-04111 and RGPIN-2025-04076.  The research of JS was supported by NSERC, grant RGPIN-2024-03725.

\end{document}